\def\version{10/06/2009 version 4 --- To appear in \emph{Alpine perspectives on algebraic topology}}

\documentclass[a4paper,11pt]{amsart}  \usepackage{fullpage}

\usepackage{amssymb,amscd,amsxtra}
\usepackage[all]{xy}
\usepackage{mathrsfs}

\theoremstyle{plain}
\newtheorem{thm}{Theorem}[section]
\newtheorem{lem}[thm]{Lemma}
\newtheorem{prop}[thm]{Proposition}
\newtheorem{cor}[thm]{Corollary}

\theoremstyle{definition}
\newtheorem{rem}[thm]{Remark}
\newtheorem{defn}[thm]{Definition}
\newtheorem{examp}[thm]{Example}

\numberwithin{equation}{section}

\def\ie{\emph{i.e.}}
\def\ds{\displaystyle}
\def\:{\colon}
\def\.{\cdot}

\def\<{\left\langle}
\def\>{\right\rangle}
\def\({\left(}
\def\){\right)}
\def\ph#1{\phantom{#1}}
\def\epsilon{\varepsilon}
\def\leq{\leqslant}
\def\geq{\geqslant}

\def\lra{\longrightarrow}
\def\Lra{\Longrightarrow}
\def\ra{\rightarrow}
\def\hat#1{\widehat{#1}}
\def\tilde#1{\widetilde{#1}}
\def\iso{\cong}

\DeclareMathOperator{\coker}{coker}

\DeclareMathOperator{\im}{im}

\def\F{\mathbb{F}}

\def\k{\Bbbk}

\def\Q{\mathbb{Q}}
\def\N{\mathbb{N}}

\def\Z{\mathbb{Z}}
\def\hZ{\hat{\mathbb{Z}}}

\def\widebar#1{\overline{#1}}
\def\ideal{\triangleleft}

\DeclareMathOperator{\End}{End}

\DeclareMathOperator*{\Gal}{Gal}
\DeclareMathOperator{\Hom}{Hom}

\DeclareMathOperator{\Tor}{Tor}

\def\id{\mathrm{id}}
\def\Id{\mathrm{Id}}
\DeclareMathOperator*{\colim}{colim}

\def\nr{\mathrm{nr}}

\def\Fc{\bar{\F}}
\DeclareMathOperator{\Char}{char}
\def\phi{\varphi}

\def\hotimes{{\hat\otimes}}

\DeclareMathOperator{\Map}{Map}
\def\Mapc{\Map^{\mathrm{c}}}
\def\nr{\mathrm{nr}}
\def\Enr{E^\nr}

\def\Gnr{\mathbb{G}^\nr}

\newdir{ >}{{}*!/-8pt/@{>}}

\begin{document}
\title[$L$-complete Hopf algebroids]
     {$L$-complete Hopf algebroids and their comodules}
\author{Andrew Baker}

\address{
Department of Mathematics, University of Glasgow,
Glasgow G12 8QW, Scotland.}
\email{a.baker@maths.gla.ac.uk}
\urladdr{http://www.maths.gla.ac.uk/$\sim$ajb}
\thanks{The author was partially funded by a YFF Norwegian
Research Council grant whilst a visiting Professor at the
University of Oslo, and by an EPSRC Research Grant EP/E023495/1.
I would like to thank the Oslo topologists for their support
and interest in the early stages of this work, Mark Hovey,
Uli Kr\"ahmer and Geoffrey Powell for helpful conversations,
and finally the referee for perceptive comments and suggesting
some significant improvements in our exposition.
\\[7pt]
This paper is dedicated to the mountains of the Arolla valley}
\date{\version\hfill \texttt{arxiv.org:0901.1471}}
\begin{abstract}
We investigate Hopf algebroids in the category of $L$-complete
modules over a commutative Noetherian regular complete local
ring. The main examples of interest in algebraic topology are
the Hopf algebroids associated to Lubin-Tate spectra in the
$K(n)$-local stable homotopy category, and we show that these
have Landweber filtrations for all finitely generated discrete
modules.

Along the way we investigate the canonical Hopf algebras
associated to Hopf algebroids over fields and introduce a
notion of unipotent Hopf algebroid generalising that for
Hopf algebras.
\end{abstract}

\maketitle

\section*{Introduction}

In this paper we describe some algebraic machinery that has
been found useful occurs when working with the $K(n)$-local
homotopy category, and specifically the cooperation structure
on covariant functors of the form $E^\vee_*(-)$, where
\[
E^\vee_*(X) = \pi_*(L_{K(n)}(E\wedge X))
\]
is the homotopy of the Bousfield localisation of $E\wedge X$
with respect to Morava $K$-theory $K(n)$. Our main focus is
on algebra, but our principal examples originate in stable
homotopy theory.

In studying the $K(n)$-local homotopy category, topologists
have found it helpful to use the notion of \emph{$L$-complete
module} introduced for other purposes by Greenlees and May
in~\cite{JPCG-JPM}. It is particularly fortunate that the
Lubin-Tate spectrum $E_n$ associated with a prime $p$ and
$n\geq1$ has for its homotopy ring
\[
\pi_*E_n = W\F_{p^n}[[u_1,\ldots,u_{n-1}]][u,u^{-1}],
\]
where all generators are in degree $0$ except $u$ which has
degree~$2$. Thus (apart form the odd-even grading) the
coefficient ring for the covariant functor $(E_n)^\vee_*(-)$
is a commutative Noetherian regular complete local ring of
dimension~$n$ and the theory of $L$-complete modules works
well. For details of these applications see~\cite{MH-NS:666},
also~\cite{Ho:colim} and~\cite[section~7]{AB&BR}.

In this paper we consider analogues of Hopf algebroids in
the category of $L$-complete modules over a commutative
Noetherian regular complete local ring, and relate this to
our earlier work of~\cite{AB-LFT}. Since the latter appeared
there has been a considerable amount of work by Hovey and
Strickland~\cite{MH-NS:Comod&LET,MH-NS:LocalBP*BP} on
localisations of categories of comodules over $MU_*MU$ and
$BP_*BP$, but that seems to be unrelated to the present theory.
One of our main motivations was to try to understand the precise
sense in which the $L$-complete theory differs from these other
theories and we intend to return to this in future work.

We introduce a notion of unipotent Hopf algebroid over a field
and then consider the relationship between modules over a Hopf
algebroid $(\boldsymbol{k},\Gamma)$ and over its associated Hopf
algebra $(\boldsymbol{k},\Gamma')$ and unicursal Hopf algebroid.
We show that if $(\boldsymbol{k},\Gamma')$ is unipotent then so
is $(\boldsymbol{k},\Gamma)$. As a consequence a large class of
Hopf algebroids over local rings have composition series for
finitely generated comodules which are discrete in the sense
that they are annihilated by some power of the maximal ideal.

We end by discussing the important case $E^\vee_*E$ for a
Lubin-Tate spectrum $E$. In particular we verify that finitely
generated comodules over this $L$-complete Hopf algebroid have
Landweber filtrations.

For completeness, in two appendices we continue the discussion
of the connections with twisted group rings begun in~\cite{AB-LFT},
and expand on a result of~\cite{Ho:colim} on the non-exactness
of coproducts of $L$-complete modules.

\section{$L$-complete modules}\label{sec:L-modules}

Let $(R,\mathfrak{m})$ be a commutative Noetherian regular local
ring, and let $n=\dim R$. We denote the category of (left)
$R$-modules by $\mathscr{M}=\mathscr{M}_R$. Undecorated tensor
products will be taken over $R$, \ie, $\otimes=\otimes_R$. We
will often write $\hat{R}$ for the $\mathfrak{m}$-adic completion
$R\sphat_{\mathfrak{m}}$, and $\hat{\mathfrak{m}}$ for
$\mathfrak{m}\sphat_{\mathfrak{m}}$.

The $\mathfrak{m}$-adic completion functor
\[
M\mapsto M\sphat_{\mathfrak{m}}
\]
on $\mathscr{M}$ is neither left nor right exact. Following~\cite{JPCG-JPM},
we consider its left derived functors $L_s=L_s^{\mathfrak{m}}$
($s\geq0$). We recall that there are natural transformations
\[
\Id \xrightarrow{\ \eta\ } L_0 \lra (-)\sphat_{\mathfrak{m}}
                                  \lra R/\mathfrak{m}\otimes_R(-).
\]
The two right hand natural transformations are epimorphic for each
module, and $L_0$ is idempotent, \ie, $L_0^2\iso L_0$. It is also
true that $L_s$ is trivial for $s>n$. For computing the derived
functor for an $R$-module $M$ and $s>0$ there is a natural exact
sequence of~\cite[proposition~1.1]{JPCG-JPM}:
\begin{equation}\label{eqn:Ls-exactseq}
0\ra {\lim_k}^1\Tor^R_{s+1}(R/\mathfrak{m}^k,M) \lra L_sM
    \lra \lim_k\Tor^R_s(R/\mathfrak{m}^k,M) \ra 0.
\end{equation}

It is an important fact that tensoring with finitely generated
modules interacts well with the functor $L_0$. A module is said
to have \emph{bounded $\mathfrak{m}$-torsion module} if it is
annihilated by some power of $\mathfrak{m}$.
\begin{prop}\label{prop:fg-L0}
Let $M,N$ be $R$-modules with $M$ finitely generated.
Then there is a natural isomorphism
\[
M\otimes L_0N \lra L_0(M\otimes N).
\]
In particular,
\begin{align*}
L_0M &\iso M\sphat_{\mathfrak{m}} \iso \hat{R}\otimes M, \\
R/\mathfrak{m}^k\otimes L_0N
  &\iso R/\mathfrak{m}^k\otimes N = N/\mathfrak{m}^kN.
\end{align*}
Hence, if $N$ is a bounded $\mathfrak{m}$-torsion module
then it is $L$-complete.
\end{prop}
\begin{proof}
See \cite[proposition~A.4]{MH-NS:666}.
\end{proof}

A module $M$ is said to be \emph{$L$-complete} if
$\eta\:M\lra L_0M$ is an isomorphism. The subcategory of
$L$-complete modules $\hat{\mathscr{M}}\subseteq\mathscr{M}$
is a full subcategory and the functor
$L_0\:\mathscr{M}\lra\hat{\mathscr{M}}$ is left adjoint
to the inclusion $\hat{\mathscr{M}}\lra\mathscr{M}$. The
category $\hat{\mathscr{M}}$ has projectives, namely the
\emph{pro-free modules} which have the form
\[
L_0F = F\sphat_{\mathfrak{m}}
\]
for some free $R$-module $F$. Thus $\hat{\mathscr{M}}$ has enough
projectives and we can do homological algebra to define derived
functors of right exact functors.

By \cite[theorem~A.6(e)]{MH-NS:666}, the category $\hat{\mathscr{M}}$
is abelian and has limits and colimits
which are obtained by passing to $\mathscr{M}$, taking (co)limits
there and applying $L_0$. For the latter there are non-trivial
derived functors which by~\cite{Ho:colim} satisfy
\[
{\colim_{\hat{\mathscr{M}}}}^{s} = L_s\colim_{\mathscr{M}},
\]
so ${\colim}^{s}$ is trivial for $s>n$. In fact, for a coproduct
$\coprod_\alpha M_\alpha$ with $M_\alpha\in\hat{\mathscr{M}}$,
we also have
\[
L_n\biggl(\coprod_\alpha M_\alpha\biggr) = 0.
\]

The category $\hat{\mathscr{M}}$ has a symmetric monoidal
structure coming from the tensor product in $\mathscr{M}$.
For $M,N\in\mathscr{M}$,
let
\[
M\hotimes N=L_0(M\otimes N)\in\hat{\mathscr{M}}.
\]
Note that we also have
\[
M\hotimes N \iso L_0(L_0M\otimes L_0N).
\]
As in~\cite{MH-NS:666}, we find that $(\hat{\mathscr{M}},\hotimes)$
is a symmetric monoidal category.

For any $R$-module $M$, there are natural homomorphisms
\[
\hat R\otimes L_0M \lra L_0(R \otimes M) \lra L_0M,
\]
so we can view $\hat{\mathscr{M}}$ as a subcategory of
$\mathscr{M}_{\hat R}$; since $\hat R$ is a flat $R$-algebra,
for many purposes it is better to think of $\hat{\mathscr{M}}$
this way. For example, the functor $L_0=L^{\mathfrak{m}}_0$
on $\mathscr{M}$ can be expressed as
\[
L^{\mathfrak{m}}_0M \iso L^{\hat{\mathfrak{m}}}_0(\hat R\otimes M),
\]
where $L^{\hat{\mathfrak{m}}}_0$ is the derived functor on the
category of $\hat{R}$-modules $\mathscr{M}_{\hat R}$ associated
to completion with respect to the induced ideal
$\hat{\mathfrak{m}}\ideal\hat R$. Finitely generated modules
over $\hat R$ (which always lie in $\hat{\mathscr{M}}$) are
completions of finitely generated $R$-modules.

There is an analogue of Nakayama's Lemma provided
by~\cite[theorem~A.6(d)]{MH-NS:666}.
\begin{prop}\label{prop:Nakayama}
For $M\in\hat{\mathscr{M}}$,
\[
M=\mathfrak{m}M\quad\Lra\quad M=0.
\]
\end{prop}

This can be used to give proofs of analogues of many standard
results in the theory of finitely generated modules over
commutative rings. For example,
\begin{cor}\label{cor:Nakayama1}
Let $M\in\hat{\mathscr{M}}$ and suppose that $N\subseteq M$
is the image of a morphism $N'\lra M$ in $\hat{\mathscr{M}}$.
Then
\[
M=N+\mathfrak{m}M \quad\Lra\quad N=M.
\]
\end{cor}
\begin{proof}
The standard argument works here since we can form $M/N$
in $\hat{\mathscr{M}}$ and as $M/N = \mathfrak{m}M/N$,
we have $M/N = 0$, whence $N=M$.
\end{proof}

At this point we remind the reader that over a commutative
local ring, every projective module is in fact free by a
result of Kaplansky~\cite[theorem~2.5]{Matsumura}. The proof
of our next result is similar to that of the better known but
weaker result for finitely generated projectives which is a
direct consequence of Nakayama's Lemma.
\begin{cor}\label{cor:Nakayama2}
Let $M\in\hat{\mathscr{M}}$ and suppose that $F$ is a free
module for which there is an isomorphism
$F/\mathfrak{m}F\iso M/\mathfrak{m}M$. Then there is an
epimorphism $L_0F\lra M$.
%
\end{cor}
\begin{proof}
The isomorphism $F/\mathfrak{m}F\xrightarrow{\iso}M/\mathfrak{m}M$
lifts to a map $F\lra M$ that factors through
\[
L_0F\lra L_0M\iso M,
\]
which has image $N\subseteq M$ say. There is a commutative
diagram
\[
\xymatrix{
F\ar[r]\ar[d] & L_0F\ar[r]\ar[d] & M \ar[d] \\
F/\mathfrak{m}F\ar[r]^{\iso\ph{abc}}
   &L_0F/\mathfrak{m}L_0F\ar[r]^{\ph{abc}\iso}
   & M/\mathfrak{m}M
}
\]
which shows that $M = N+\mathfrak{m}M$, so $N=M$.
%
\end{proof}

Here is another example. Let $S\subseteq\mathfrak{m}$ and
let $M = SM$ be the submodule of $M$ consisting of all sums
of elements of the form $sz$ for $s\in $ and $z\in M$. We
say that an $R$-module $M$ is \emph{$S$-divisible} if for
every $x\in M$ and $s\in S$, there exists $y\in M$ such
that $x=sy$, \ie, $M = SM$. Since $R$ is an integral domain,
this is consistent with Lam's definition in chapter~1\S3C
of~\cite{Lam}, see also corollary~(3.17)$'$.
\begin{lem}\label{lem:NoInj}
Let $M\in\hat{\mathscr{M}}$ and let $S\subseteq\mathfrak{m}$
be non-empty. If $M$ is $S$-divisible then it is trivial. In
particular, injective objects in $\hat{\mathscr{M}}$ are trivial.
\end{lem}
\begin{proof}
For the first statement, if $M=SM$ then $M\subseteq\mathfrak{m}M$
and so $M=\mathfrak{m}M$, therefore $M=0$.

Let $M$ be injective in the category $\hat{\mathscr{M}}$. Then
for each $x\in M$ there is a homomorphism $R\lra M$ for which
$1\mapsto x$. This extends to a homomorphism $L_0R=\hat{R}\lra M$.
For $s\in S$, there is a homomorphism $L_0R\lra L_0R$ induced
from multiplication by~$s$. By injectivity there is an extension
to a diagram
\[
\xymatrix{
0\ar[r] & L_0R\ar[r]^{s}\ar[d] & L_0R\ar@{.>}[dl] \\
& M &
}
\]
so $M$ is $S$-divisible.
\end{proof}

We will find it useful to know about some basic functors
on $\hat{\mathscr{M}}$ and their derived functors.

Let $N$ be an $L$-complete $R$-module. As the functors
$N\otimes(-)\:\hat{\mathscr{M}}\lra\mathscr{M}$ and
$L_0\:\mathscr{M}\lra\hat{\mathscr{M}}$ are right exact,
so is the endofunctor of $\hat{\mathscr{M}}$
\[
M\mapsto N\hotimes M = L_0(N\otimes M).
\]
Therefore we can use resolutions by projective objects
(\ie, pro-free $L$-complete modules) to form the left
derived functors, which we will denote $\widehat{\Tor}^R_s(N,-)$,
where
\[
\widehat{\Tor}^R_0(N,M)=N\hotimes M.
\]
If $P$ is pro-free then by definition, $\widehat{\Tor}^R_s(N,P)=0$
for $s>0$. On the other hand, $\widehat{\Tor}^R_s(P,-)$ need
not be the trivial functor (see Appendix~\ref{sec:App2}). This
shows that $\widehat{\Tor}^R_s(-,-)$ is not a balanced bifunctor,
\ie, in general
\[
\widehat{\Tor}^R_s(N,M)\not\cong\widehat{\Tor}^R_s(M,N).
\]

By Proposition~\ref{prop:fg-L0}, for a finitely generated
$R$-module $N_0$, $L_0N_0$ is a finitely generated
$\hat{R}$-module which induces the left exact functor
\[
M\mapsto L_0(N_0\otimes M) \iso N_0\otimes M.
\]

For $M\in\hat{\mathscr{M}}$, we can choose consider a free
resolution in $\mathscr{M}$,
\[
F_*\lra M\ra 0.
\]
Recalling that $L_0M\iso M$ and $L_sM=0$ for $s>0$, the
homology of $L_0F_*$ is
\[
H_*(L_0F_*) = L_0M \iso M,
\]
hence we have a resolution of $M$ by pro-free modules
\[
L_0F_*\lra M\ra 0.
\]
Then
\[
\widehat{\Tor}^R_*(L_0N_0,M) = H_*(L_0(N_0\otimes F_*)).
\]
But now we have
\[
L_0(N_0\otimes F_*) \iso
N_0\otimes L_0F_* = N_0\otimes (F_*)\sphat_{\mathfrak{m}}.
\]

When $N$ is a finitely generated $\mathfrak{m}$-torsion
module, we have $L_0N=N$ and
\[
L_0(N\otimes F_*) \iso N\otimes F_*,
\]
therefore
\begin{equation}\label{eqn:hatTor-NM}
\widehat{\Tor}^R_*(N,M) = \Tor^R_*(N,M).
\end{equation}
Now take a free resolution
\[
P_* \lra N \ra0
\]
with each $P_s$ finitely generated. Then
\[
\widehat{\Tor}^R_*(M,N) = H_*(L_0(M\otimes P_*))
            \iso H_*(M\otimes P_*)) = \Tor^R_*(M,N),
\]
hence
\begin{equation}\label{eqn:hatTor-MN}
\widehat{\Tor}^R_*(M,N) \iso \Tor^R_*(M,N).
\end{equation}

Combining \eqref{eqn:hatTor-NM} and \eqref{eqn:hatTor-MN}, we
obtain the following restricted result on $\widehat{\Tor}^R_*$
as a balanced bi-functor. As far as we can determine, there is
no general analogue of this for arbitrary $L$-complete modules~$N$
which are finitely generated as $\hat{R}$-modules.
\begin{prop}\label{prop:hTor-Balanced}
Let $M,N$ be $L$-complete $R$-modules, where $N$ is a finitely
generated\/ $\mathfrak{m}$-torsion module. Then
\[
\widehat{\Tor}^R_*(M,N) \iso \Tor^R_*(M,N)
                    \iso \widehat{\Tor}^R_*(N,M).
\]
\end{prop}

When $N$ is a finitely generated $\mathfrak{m}$-torsion module,
we may also consider the composite functor
$\mathscr{M}\lra\hat{\mathscr{M}}$ for which
\[
M\mapsto L_0(N\otimes M).
\]
Since
\[
L_0(N\otimes M) = N\otimes L_0M = N\otimes M,
\]
this functor has for its left derived functors $\Tor^R_*(N,-)$
and there is an associated composite functor spectral sequence.
\begin{prop}\label{prop:Tor-L-SS}
Let $N$ be a finitely generated $\mathfrak{m}$-torsion module.
Then for each $R$-module $M$, there is a natural first quadrant
spectral sequence
\[
\mathrm{E}^2_{s,t} = \widehat{\Tor}^R_s(N,L_tM)
                   = \Tor^R_s(N,L_tM)
\quad\Lra\quad
\Tor^R_{s+t}(N,M).
\]
\end{prop}
\begin{proof}
Let $P_*\lra N\ra 0$ and $Q_*\lra M\ra 0$ be free resolutions.
Since $R$ is Noetherian, we can assume that each $P_s$ is finitely
generated, so
\[
L_0(P_*\otimes Q_*) \iso P_*\otimes L_0Q_*.
\]
Taking first homology, then second homology, and using~\eqref{eqn:hatTor-NM}
together with the fact that each $L_0Q_t$ is projective
in $\hat{\mathscr{M}}$, we obtain
\begin{align*}
H^{II}_*H^I_*(P_*\otimes L_0Q_*)
        &= H^{II}_*\Tor^R_*(N,L_0Q_*) \\
        &= H^{II}_*\widehat{\Tor}^R_*(N,L_0Q_*) \\
        &= H^{II}_*\widehat{\Tor}^R_0(N,L_0Q_*) \\
        &= H^{II}_*(N\otimes L_0Q_*) \\
        &= H^{II}_*(N\otimes Q_*)  \\
        &= \Tor^R_*(N,M).
\end{align*}
The resulting spectral sequence collapses at its $\mathrm{E}^2$-term.
Taking second homology then first homology we obtain
\begin{align*}
H^I_*H^{II}_*(P_*\otimes L_0Q_*)
             &= H^I_*(P_*\otimes L_*M) \\
             &= \Tor^R_*(N,L_*M).
\end{align*}
This is the $\mathrm{E}^2$-term of a spectral sequence converging
to $\Tor^R_*(N,M)$ as claimed.
\end{proof}

\begin{lem}\label{lem:flat->pro-free}
Let $M$ be a flat $R$-module. Then
\[
L_sM =
\begin{cases}
M\sphat_{\mathfrak{m}} &\text{\rm if $s=0$}, \\
\ph{a}0 & \text{\rm otherwise},
\end{cases}
\]
and $L_0M$ is pro-free.
\end{lem}
\begin{proof}
For each $s\geq0$, the exact sequence of~\eqref{eqn:Ls-exactseq}
and the flatness of $M$ yield
\[
L_sM =
\begin{cases}
M\sphat_{\mathfrak{m}} &\text{if $s=0$}, \\
\ph{a}0 & \text{otherwise}.
\end{cases}
\]
The spectral sequence of Proposition~\eqref{prop:Tor-L-SS}
with $N=R/\mathfrak{m}$ degenerates so that for each
$s>0$ we obtain
\[
\Tor^R_s(R/\mathfrak{m},L_0M)=\Tor^R_s(R/\mathfrak{m},M)=0,
\]
hence $L_0M$ is pro-free by~\cite[theorem~A.9(3)]{MH-NS:666}.
\end{proof}

If $M$ is a finitely generated $R$-module then it has bounded
$\mathfrak{m}$-torsion, hence by~\cite[theorem~1.9]{JPCG-JPM},
$L_0M=M\sphat_{\mathfrak{m}}$ and $L_sM=0$ for $s>0$. More
generally, if $F$ is a free module, then $F\otimes M$ has
bounded $\mathfrak{m}$-torsion, so
\begin{align*}
L_s(F\otimes M) =
\begin{cases}
F\sphat_{\mathfrak{m}}\otimes M & \text{if $s=0$}, \\
\ph{000}0 & \text{if $s\neq0$}.
\end{cases}
\end{align*}
If we choose a basis for $F$, we can write $F=\bigoplus_\alpha R$,
and the last observation amounts to the vanishing of the higher
derived functors of the coproduct functor
\[
\hat{\mathscr{M}}_{\mathrm{fg}} \lra \hat{\mathscr{M}};
\quad
M\mapsto L_0\biggl(\bigoplus_\alpha M_\alpha\biggr)
\]
defined on the subcategory of completions of finitely generated
modules (which is the same as the subcategory of finitely generated
$\hat R$\;-modules). Now for any resolution $P_*\lra M\ra0$ of
a finitely generated module $M$ by finitely generated projectives,
we have
\[
L_0(F\otimes P_*) \iso L_0F\otimes P_*.
\]
The left hand side has as its homology the above derived
functors, so
\[
H_*L_0(F\otimes P_*) = L_0(F\otimes M) = L_0F\otimes M,
\]
while the right hand side has homology
\[
H_*(L_0F\otimes P_*) = \Tor^R_*(L_0F,M).
\]
So for $s>0$, $\Tor^R_s(L_0F,M) = 0$.

For $P\in\hat{\mathscr{M}}$, the functor on $\hat{\mathscr{M}}$
given by $M\mapsto P\hotimes M$ is right exact. We say that $P$
is \emph{$L$-flat} if the functor $P\hotimes(-)$ is exact on
$\hat{\mathscr{M}}$. However, the $L$-flat modules are easily
identified, at least when $n=\dim R=1$, because of the
following result.
\begin{prop}\label{prop:Flat->profree}
Let $P\in\hat{\mathscr{M}}$ be $L$-flat. Then $P$ is pro-free.
\end{prop}
\begin{proof}
The proof is similar to that of Corollary~\ref{cor:Nakayama2}
and is based on a standard argument for finitely presented
flat modules over a local ring. Choose a free $R$-module $F$
for which $F/\mathfrak{m}F\iso P/\mathfrak{m}P$. If $f\:F\lra P$
is a homomorphism covering this isomorphism, there is an
extension to a homomorphism $\hat{f}\:F\sphat_{\mathfrak{m}}\lra P$.
Then we have
\[
\im\hat{f}+\mathfrak{m}P=P
\]
and so
\[
\mathfrak{m}(P/\im\hat{f})=P/\im\hat{f},
\]
hence $\im\hat{f}=P$ by Nakayama's Lemma.

Let $K=\ker\hat{f}$. Tensoring the exact sequence
\[
0\ra K\lra F\sphat_{\mathfrak{m}}\lra P \ra 0
\]
with $R/\mathfrak{m}$, by flatness of $P$ we obtain the
exact sequence
\[
0\ra K/\mathfrak{m}K\lra F/\mathfrak{m}F
              \xrightarrow{\iso} P/\mathfrak{m}P \ra 0
\]
so $K/\mathfrak{m}K=0$. Hence $K=0$ by Nakayama's Lemma.
\end{proof}

In general, tensoring with a pro-free module need not be
left exact on $\hat{\mathscr{M}}$ as is shown by an example
in Appendix~\ref{sec:App2}. In particular, when $n>1$,
infinitely generated pro-free modules may not always be
flat. Instead we can restrict attention to $L$-flatness
on restricted classes of $L$-complete modules. We say that
$P$ is \emph{weakly $L$-flat} if the functor
\[
{\hat{\mathscr{M}}}_{\mathrm{bt}}\lra{\hat{\mathscr{M}}};
\quad
M\mapsto P\hotimes M
\]
is exact on the subcategory ${\hat{\mathscr{M}}}_{\mathrm{bt}}$
of bounded $\mathfrak{m}$-torsion modules. Then if $Q$ is a flat
module, $L_0Q$ is weakly $L$-flat since for any
$N\in{\hat{\mathscr{M}}}_{\mathrm{bt}}$,
\[
L_0Q\hotimes N \iso L_0(Q\otimes N) \iso Q\hotimes N.
\]

\section{$L$-complete Hopf algebroids}\label{sec:L-HA}

To ease notation, from now on we assume that $(R,\mathfrak{m})$
is a commutative Noetherian regular local ring which is
$\mathfrak{m}$-adically complete, \ie, $R=\hat{R}=R\sphat_{\mathfrak{m}}$.
We assume that $R$ is an algebra over some chosen local subring
$(\k_0,\mathfrak{m}_0)$ so that the inclusion map is local, \ie,
$\mathfrak{m}_0=\k_0\cap\mathfrak{m}$. We write $\k=R/\mathfrak{m}$
for the residue field.

Let $\Gamma\in\mathscr{M}_{\k_0}$. We need to assume extra structure
on $\Gamma$ to define the notion of an \emph{$L$-complete Hopf
algebroid}. Unfortunately this is quite complicated to describe.

%

A (non-unital) \emph{ring object} $A\in\mathscr{M}_{\k_0}$ is
equipped with a product morphism $\phi\:A\otimes_{\k_0}A\lra A$
which is associative, \ie, the following diagram commutes.
\[
\xymatrix{
A\otimes_{\k_0} A \otimes_{\k_0} A \ar[rr]^{\ph{ab}\id\otimes\phi}\ar[d]_{\id\otimes\phi}
                                       && A\otimes_{\k_0} A\ar[d]^{\phi} \\
A\otimes_{\k_0} A\ar[rr]^{\phi} && A
}
\]
It is commutative if
\[
\xymatrix{
A \otimes_{\k_0} A \ar[rr]^{\ph{ab}\mathrm{switch}}\ar[dr]_{\phi}
                                       && A\otimes_{\k_0}A\ar[ld]^{\phi} \\
 & A &
}
\]
also commutes. An \emph{$R$-unit} for $\phi$ is a $\k_0$-algebra
homomorphism $\eta\:R\lra A$.

\begin{defn}\label{defn:biunital}
A ring object is \emph{$R$-biunital} if it has two units
$\eta_L,\eta_R\:R\lra A$ which extend to give a morphism
$\eta_L\otimes\eta_R\:R\otimes_{\k_0}R\lra A$.
\end{defn}

To distinguish between the two $R$-module or $R$-module
structures on $A$, we will write ${}_RA$ and $A_R$. When
discussing $A_R$ we will emphasise the use of the right
module structure whenever it occurs. In particular, from
now on tensor products over $R$ are to be interpreted as
bimodule tensor products ${}_R\otimes_R$, even though we
often write $\otimes$.

\begin{defn}\label{defn:biunital-Lcomplete}
An $R$-biunital ring object $A$ is \emph{$L$-complete} if
$A$ is $L$-complete as both a left and a right $R$-module.
\end{defn}

\begin{defn}\label{defn:L-compHA}
Suppose that $\Gamma$ is an $L$-complete commutative $R$-biunital
ring object with left and right units $\eta_L,\eta_R\:R\lra\Gamma$,
and has the following additional structure:
\begin{itemize}
\item
a \emph{counit}: a $\k_0$-algebra homomorphism $\epsilon\:\Gamma\lra R$;
\item
a \emph{coproduct}: a $\k_0$-algebra homomorphism
$\psi\:\Gamma\lra\Gamma\hotimes\Gamma=\Gamma_R\hotimes{}_R\Gamma$;
\item
an \emph{antipode}: a $\k_0$-algebra homomorphism
$\chi\:\Gamma\lra\Gamma$.
\end{itemize}
Then $\Gamma$ is an \emph{$L$-complete Hopf algebroid} if
\begin{itemize}
\item
with this structure, $\Gamma$ becomes a cogroupoid object,
\item
if $\Gamma$ is pro-free as a left (or equivalently as a right)
$R$-module,
\item
the ideal $\mathfrak{m}\lhd R$ is \emph{invariant}, \ie,
$\mathfrak{m}\Gamma=\Gamma\mathfrak{m}$.
\end{itemize}
We often denote such a pair by $(R,\Gamma)$ when the structure
maps are clear.
\end{defn}

The cogroupoid condition is essentially the same as that spelt
out in~\cite[definition~A1.1.1]{DCR:Book2} by interpreted in
the the context of $L$-complete bimodules. In particular we
have a relationship between the two notions of $L$-completeness
for $\Gamma$ since the antipode $\chi$ induces an isomorphism
of $R$-modules $\chi\:{}_R\Gamma\iso\Gamma_R$. The pro-freeness
condition is a disguised version of flatness required to do
homological algebra, and

\begin{defn}\label{defn:L-HAcomodule}
Let $(R,\Gamma)$ be an $L$-complete Hopf algebroid and let
$M\in\hat{\mathscr{M}}$. Then an $R$-module homomorphism
$\rho\:M\lra\Gamma\hotimes M$ makes $M$ into a \emph{left
$(R,\Gamma)$-comodule} or \emph{$\Gamma$-comodule} if the
following diagrams commute.
\[
\xymatrix{
M \ar[rr]^{\ph{ab}\rho}\ar[d]_{\rho}
                       && \Gamma\hotimes M\ar[d]^{\psi\otimes\id} \\
\Gamma\hotimes M\ar[rr]^{\rho\otimes\id} && \Gamma\hotimes\Gamma\hotimes M
}
\qquad
\xymatrix{
M\ar[r]^{\rho}\ar[rd]_{\iso} & \Gamma\hotimes M\ar[d]^{\epsilon\otimes\id} \\
 & R\hotimes M
}
\]
There is a similar definition of a \emph{right $\Gamma$-comodule}.
\end{defn}

Let $(R,\Gamma)$ be an $L$-complete Hopf Algebroid. Then given
a morphism of $\Gamma$-comodules $\theta\:M\lra N$, there is
a commutative diagram of solid arrows
\[
\xymatrix{
0\ar[r] & \ker\theta \ar[r]\ar@{.>}[d] & M \ar[r]^{\theta}\ar[d]^{\psi}
                                                          & N\ar[d]^{\psi}\\
        & \Gamma\hotimes\ker\theta \ar[r]
        & \Gamma\hotimes M \ar[r]^{\id\otimes\theta} & \Gamma\hotimes N
}
\]
but if $\id\otimes\theta$ is not a monomorphism then the dotted
arrow may not exist or be unique. If $\Gamma\hotimes(-)$ always
preserved exactness then this would not present a problem, but
this is not so easily ensured in great generality.

If $\Gamma$ is pro-free, then as already noted, $\Gamma\hotimes(-)$
is exact on the categories ${\hat{\mathscr{M}}}_{\mathrm{bt}}$ and
$\hat{\mathscr{M}}_{\mathrm{fg}}$, so in each of these contexts
the above diagram always has a completion by a unique dotted arrow.
Therefore the categories of $\Gamma$-comodules in
${\hat{\mathscr{M}}}_{\mathrm{bt}}$ and $\hat{\mathscr{M}}_{\mathrm{fg}}$
are abelian since they have kernels and all the other axioms are
satisfied.

\begin{examp}\label{examp:HA->L-HA}
Let $(R,\Gamma)$ be a flat Hopf algebroid over the commutative
Noetherian regular local ring~$R$, and assume that
$\mathfrak{m}\Gamma=\Gamma\mathfrak{m}$. By
Lemma~\ref{lem:flat->pro-free},
\[
L_0({}_R\Gamma) = \Gamma\sphat_{\mathfrak{m}} = L_0(\Gamma_R),
\]
where $\Gamma\sphat_{\mathfrak{m}}$ denotes the completion
with respect to the ideal $\mathfrak{m}\Gamma$ which equals
$\Gamma\mathfrak{m}$.
\end{examp}

\begin{defn}\label{defn:Invtideal}
Let $(R,\Gamma)$ be a Hopf algebroid over a local ring
$(R,\mathfrak{m})$ or an $L$-complete Hopf algebroid.
\begin{itemize}
\item
The maximal ideal $\mathfrak{m}\lhd R$ is \emph{invariant}
if $\mathfrak{m}\Gamma=\Gamma\mathfrak{m}$. More generally,
a subideal $I\subseteq\mathfrak{m}$ is \emph{invariant}
if $I\Gamma=\Gamma I$.
\item
An $(R,\Gamma)$-comodule $M$ is \emph{discrete} if for each
element $x\in M$, there is a $k\geq1$ for which
$\mathfrak{m}^kx=\{0\}$; if $M$ is also finitely generated
as an $R$-module, then~$M$ is discrete if and only if there
is a $k_0$ such that $\mathfrak{m}^{k_0}M=\{0\}$.
\item
An $(R,\Gamma)$-comodule $M$ is \emph{finitely generated}
if it is finitely generated as an $R$-module.
\end{itemize}
\end{defn}

If~$M$ is a $(R,\Gamma)$-comodule, then for any invariant
ideal~$I$, $IM\subseteq M$ is a subcomodule.

If $(R,\Gamma)$ be a (possibly $L$-complete) Hopf algebroid
for which $\mathfrak{m}$ is invariant, then
$(\k,\Gamma/\mathfrak{m}\Gamma)$ is a Hopf algebroid over
the residue field $\k$. If a $\Gamma$-comodule is annihilated
by $\mathfrak{m}$ then it is also a
$\Gamma/\mathfrak{m}\Gamma$-comodule.

\section{Unipotent Hopf algebroids}\label{sec:UnipotentHA}

We start by recalling the notion of a \emph{unipotent} Hopf
algebra $H$ over a field $\boldsymbol{k}$ which can be found
in~\cite{Waterhouse}. This means that every $H$-comodule $V$
which is a finite dimensional $\boldsymbol{k}$-vector space
has primitive elements, or equivalently (by the Jordan-H\"older
theorem) it admits a composition series, \ie, a finite length
filtration by subcomodules
\begin{equation}\label{eqn:JH-filt}
V=V_m\supset V_{m-1}\supset\cdots\supset V_1\supset V_0=\{0\}
\end{equation}
with irreducible quotient comodules $V_k/V_{k+1}\iso\boldsymbol{k}$.
In particular, notice that $\boldsymbol{k}$ is the only finite
dimensional irreducible $H$-comodule. Reinterpreting $H$-comodules
as $H^*$-modules where $H^*$ is the $\boldsymbol{k}$-dual of
$H$, this implies that $H^*$ is a local ring, \ie, its
augmentation ideal is its unique maximal left ideal and
therefore agrees with its Jacobson radical.

Now given a Hopf algebroid $(R,\Gamma)$ over local ring
$(R,\mathfrak{m})$ with residue field $\k=R/\mathfrak{m}$
and invariant maximal ideal $\mathfrak{m}$, the resulting
Hopf algebroid $(\k,\Gamma/\mathfrak{m}\Gamma)$ need not be
a Hopf algebra. However, we can still make the following
definition.
\begin{defn}\label{defn:HA-unipotent0}
Let $(\boldsymbol{k},\Sigma)$ be a Hopf algebroid over a
field $\boldsymbol{k}$. Then $\Sigma$ is \emph{unipotent}
if every non-trivial finite dimensional $\Sigma$-comodule
$V$ has non-trivial primitives. Hence $\boldsymbol{k}$ is
the only irreducible $\Sigma$-comodule and every finite
dimensional comodule admits a composition series as
in~\eqref{eqn:JH-filt}.
\end{defn}

In the next result we make use of Definition~\ref{defn:Invtideal}.
\begin{thm}\label{thm:Main0}
Let $(R,\Gamma)$ be a Hopf algebroid over a Noetherian local
ring $(R,\mathfrak{m})$ for which $\mathfrak{m}$ is invariant
and suppose that $(\k,\Gamma/\mathfrak{m}\Gamma)$ is a unipotent
Hopf algebroid over the residue field~$\k$. Let $M$ be a
non-trivial finitely generated discrete $(R,\Gamma)$-comodule.
Then $M$ admits a finite-length filtration by subcomodules
\[
M=M_\ell\supset M_{\ell-1}\supset\cdots\supset M_1\supset M_0=\{0\}
\]
with trivial quotient comodules $M_k/M_{k+1}\iso\k$.
\end{thm}

See~\cite{AB-LFT} for a precursor of this result. We will
refer to such filtrations as \emph{Landweber filtrations}.
\begin{proof}
The proof is similar to that used in~\cite{AB-LFT}. The idea
is to consider the descending sequence
\[
M\supseteq\mathfrak{m}M
     \supseteq\cdots\supseteq\mathfrak{m}^kM\supseteq\cdots
\]
which must eventually reach~$0$. So for some $k_0$,
$\mathfrak{m}^{k_0}M=0$ and $\mathfrak{m}^{k_0-1}M\neq0$.
The subcomodule
\[
\mathfrak{m}^{k_0-1}M\iso\mathfrak{m}^{k_0-1}M/\mathfrak{m}^{k_0}M
\]
becomes a comodule over $(\k,\Gamma/\mathfrak{m}\Gamma)$
and so it has non-trivial primitives since
$(\k,\Gamma/\mathfrak{m}\Gamma)$ is unipotent, and these are
also primitives with respect to $\Gamma$. Considering the
quotient $M/PM$, where $PM$ is the submodule of primitives,
now we can use induction on the length of a composition series
to construct the required filtration. Note that since $R$ is
local, its only irreducible module is its residue field~$\k$
which happens to be a comodule.
\end{proof}

Ravenel~\cite{DCR:Book2} introduced the associated Hopf
algebra $(A,\Gamma')$ to a Hopf algebroid $(A,\Gamma)$. When
the coefficient ring $A$ is a field, the relationship between
comodules over these two Hopf algebroids turns out to be
tractable as we will soon see.

Our next result provides a criterion for establishing when
a Hopf algebra is unipotent. We write $\otimes$ for
$\otimes_{\boldsymbol{k}}$.
\begin{lem}\label{lem:Unipotent}
Let $(H,\boldsymbol{k})$ be a Hopf algebra over a field. \\
\emph{(a)}
Suppose that
\[
\boldsymbol{k}=H_0\subseteq H_1 \subseteq\cdots\subseteq H_n
            \subseteq \cdots\subseteq H
\]
is an increasing sequence of $\boldsymbol{k}$-subspaces for
which $H=\bigcup_n H_n$ and
\[
\psi H_n \subseteq
H_0\otimes H_n + H_1\otimes H_{n-1} + \cdots + H_n\otimes H_0.
\]
Then $H$ is unipotent. Furthermore, the $H_n$ can be chosen
to be maximal and satisfy
\[
H_rH_s\subseteq H_{r+s}
\]
for all $r,s$.   \\
\emph{(b)}
Suppose that $H$ has a filtration as in \emph{(a)} and let $W$
be a non-trivial left $H$-comodule which is finite dimensional
over $\boldsymbol{k}$ and has coaction $\rho\:W\lra H\otimes W$.
Defining
\[
W_k = \rho^{-1}(H_k\otimes W)\subseteq W,
\]
we obtain an exhaustive strictly increasing filtration of $W$
by subcomodules
\[
\{0\}=W_{-1}\subset W_0\subset W_1\subset\cdots \subset W_\ell = W.
\]
\end{lem}
\begin{proof}
(a) This is part of the theorem of~\cite[section~8.3]{Waterhouse}.
The proof actually shows that the filtration by subspaces defined
in (b) is strictly increasing. \\
(b) The fact that $W_k$ is a subcomodule follows by comparing the
two sides of the equation
\[
(\Id\otimes\rho)\rho(w)=\rho(\psi\otimes\Id)\rho(w)
\]
for $w\in W_k$. Thus if we choose a basis $t_1,\ldots,t_d$ for $H_k$
and write
\[
\rho(w)= \sum_j t_j\otimes w_j
\]
for some $w_i\in W_k$, then for suitable $a_{i,r,s}\in\boldsymbol{k}$
we have
\[
\psi(t_i) = \sum_{r,s} a_{i,r,s}t_r\otimes t_s
\]
since $\psi(t_i)\in\sum_i H_i\otimes H_{k-i}\subseteq H_k\otimes H_k$.
Therefore
\[
\sum_i t_i\otimes \rho(w_i)=\sum_{j,r,s} a_{j,r,s}t_r\otimes t_s\otimes w_j,
\]
and comparing the coefficients of the left hand $t_i$, we obtain
\[
\rho(w_i) = \sum_{j,s} a_{j,i,s} t_s\otimes w_j \in H_k\otimes W.
\]
This shows that each $w_i\in W_k$, so the coproduct restricted
to $W_k$ satisfies $\rho W_k\subseteq H\otimes W_k$.
\end{proof}

\begin{examp}\label{examp:Unipotent-StAlg}
Let $p$ be an odd prime and let
\[
\mathscr{P}_* = \F_p[\zeta_k:k\geq1]
\]
be the (graded) polynomial sub-Hopf algebra of the mod~$p$
dual Steenrod algebra $\mathscr{A}_*$ with coaction
\[
\psi\zeta_n = \sum_{r=0}^n\zeta_r\otimes \zeta_{n-r}^{p^r},
\]
where $\zeta_0=1$. Then $(\mathscr{P}_*,\F_p)$ is unipotent
since the subspaces
\[
\mathscr{P}(n)_* = \F_p[\zeta_k:1\leq k\leq n]
\]
satisfy the conditions of Lemma~\ref{lem:Unipotent}. This
shows that $\mathscr{P}^*$ is a local ring.

If $p=2$, this also applies to the mod~$2$ dual Steenrod
algebra and implies that $\mathscr{A}^*$ is a local ring.
\end{examp}

For details on the next example, see the books by Ravenel and
Wilson~\cite{DCR:Book2,WSW:BPSampler}. Unfortunately the
sub-Hopf algebra $K(n)_*(E(n))$ is commonly denoted $K(n)_*K(n)$
in the earlier literature, but at the behest of the referee we
refrain from perpetuating that usage.
\begin{examp}\label{examp:Unipotent-MoravaStAlg}
Let $p$ be an odd prime and let $K(n)$ be the $n$-th $p$-primary
Morava $K$-theory. Then $K(n)_*=\F_p[v_n,v_n^{-1}]$, with
$v_n\in K(n)_{2(p^n-1)}$. There is a graded Hopf algebra over
$K(n)_*$,
\[
\Gamma(n)_* = K(n)_*(E(n)) =
K(n)_*[t_k:k\geq1]/(v_nt_\ell^{p^n}-v_n^{p^\ell}t_\ell:\ell\geq1),
\]
where $t_k\in\Gamma(n)_{2(p^k-1)}$ and $E(n)$ is a Johnson-Wilson
spectrum. In fact $\Gamma(n)_*$ is a proper sub-Hopf algebra of
$K(n)_*(K(n))$. Using standard formulae, it follows that the
$K(n)_*$-subspaces
\[
\Gamma(n,m)_* = K(n)_*(t_1,\ldots,t_m) \subseteq \Gamma_*
\]
satisfy the conditions of Lemma~\ref{lem:Unipotent}, therefore
$\Gamma(n,m)_*$ is unipotent. When $p=2$, the Hopf algebra
$\Gamma(n)_*$ is also unipotent even though $K(n)$ is not
homotopy commutative.
\end{examp}

Here is a major source of examples which includes the algebraic
ingredients used in~\cite{AB-LFT} to prove the existence of
a Landweber filtration for discrete comodules over the Hopf
algebroid of Lubin-Tate theory. For two topologised objects
$X$ and $Y$ we denote the set of all continuous maps $X\lra Y$
by $\Mapc(X,Y)$.
\begin{examp}\label{examp:dualpro-groupring}
Let $G$ be a pro-$p$-group and suppose that $G$ acts continuously
(in the sense that the action map $G\times R\lra R$ is continuous)
by ring automorphisms on $R$ which are continuous with respect to
the $\mathfrak{m}$-adic topology. Then $(R,\Mapc(G,R))$ admits
the structure of an $L$-complete Hopf algebroid, see the Appendix
of~\cite{AB-LFT} for details. This structure is dual to one on
the twisted group algebra $R[G]$. Here $\mathfrak{m}$ is invariant.
If the residue field $\k$ has characteristic~$p$, then
$(\k,\Mapc(G,R)/\mathfrak{m})$ is the continuous dual of the
pro-group ring
\[
\k[G] = \lim_{N\lhd G} \k[G/N],
\]
where $N$ ranges over the finite index normal subgroups of~$G$.
Each finite group ring $\k[G/N]$ is local since its augmentation
ideal is nilpotent, hence its only irreducible module is the
trivial module. From this it easily follows that the dual Hopf
algebra $(\k,\Mapc(G,R)/\mathfrak{m})$ is unipotent.

In each of the examples we are interested in, there is a filtration
\[
G=G_0\supset G_1\supset\cdots\supset G_k\supset G_{k-1}\supset\cdots
\]
by finite index normal subgroups $G_k\lhd G$ satisfying $\bigcap_k G_k=\{1\}$,
and the images of the natural maps
\[
\Map(G/G_k,R)\lra\Mapc(G,R)
\]
induced by the quotient maps $G\lra G/G_k$ define a filtration
with the properties listed in Lemma~\ref{lem:Unipotent}(a).
\end{examp}

\section{Unicursal Hopf algebroids}\label{sec:Unicursal}

The notion of a \emph{unicursal} Hopf algebroid $(A,\Psi)$
appeared in~\cite{DCR:Book2}, see definition~A1.1.11. It
amounts to requiring that for the subring
\[
A^\Psi = A \square_\Psi A \subseteq A \otimes_A A \iso A
\]
we have
\[
\Psi  = A\otimes_{A^\Psi} A.
\]
If $A$ is a flat $A^\Psi$-algebra then $\Psi$ is a flat
$A$-algebra. But the requirement that $A^\Psi$ is the
equalizer of the two homomorphisms $A\lra A\otimes_{A^\Psi}A$
is implied by faithful flatness, see the second theorem
of~\cite[section~13.1]{Waterhouse}.

Unicursal Hopf algebroids were introduced by Ravenel~\cite{DCR:Book2}.
However, his lemma~A1.1.13 has a correct statement for (b),
but the statement for (a) appears to be incorrect. The proofs
of (a) and (b) both appear to have minor errors or gaps. In
particular the flatness of $\Psi$ as an $A$-module is required.
Therefore we provide a slight modification of the proof given
by Ravenel. Note that we work with left rather than right
comodules.
The formulation and proof, clarifying our earlier version,
owe much to the comments of Geoffrey Powell and the referee,
particularly the relationship to descent arguments based on
faithful flatness.
\begin{lem}\label{lem:Unicursal-modules}
Let $(A,\Psi)$ be a unicursal Hopf algebroid where $A$ is
flat over $A^\Psi$. Let $M$ be a left $\Psi$-comodule. Then
there is an isomorphism of comodules
\[
M \iso A\otimes_{A^\Psi}(A\square_\Psi M),
\]
where the coaction on the right hand comodule comes from
the $\Psi$-comodule structure on $A$. In particular, if\/
$M$ is non-trivial then the primitive subcomodule\/
$A\square_\Psi M$ is non-trivial.
\end{lem}
\begin{proof}
The coaction on $M$ can be viewed as a map
$\rho\:M\lra A\otimes_{A^\Psi}M$. By coassociativity,
\[
(\Id_A\otimes\rho)\rho = (\eta_L\otimes\Id_M)\rho
                       = (\Id_A\otimes1\otimes\Id_M)\rho,
\]
hence
\[
\Id_A\otimes\rho-\Id_A\otimes1\otimes\Id_M\:
       \im\rho \lra A\otimes_{A^\Psi}A\otimes_{A^\Psi}M
\]
must be trivial.
By flatness of $A$,
\[
\xymatrix{
0\ra A\otimes_{A^\Psi}\ker(\rho-1\otimes\Id_M) \ar[r]
 & A\otimes_{A^\Psi}M \ar[rr]^{\substack{\Id_A\otimes\rho\\-\Id_A\otimes1\otimes\Id_M}\quad}
 && A\otimes_{A^\Psi}A\otimes_{A^\Psi}M
}
\]
is exact, so
\[
\im\rho\subseteq A\otimes_{A^\Psi}\ker(\rho-1\otimes\Id_M)
                      = A\otimes_{A^\Psi}(A\square_\Psi M).
\]

Since $\rho\:M\lra \Psi\otimes_A M$ is split by the augmentation
\[
\epsilon\otimes\Id_M\:\Psi\otimes_A M\lra A\otimes_A M \iso M,
\]
$\rho$ is a monomorphism. For each coaction primitive
$z\in A\square_{\Psi}M$ and $a\in A$, we have
\[
\rho(az) =  a\otimes z\in A\otimes_{A^\Psi}(A\square_{\Psi}M),
\]
hence $\im\rho=A\otimes_{A^\Psi}(A\square_{\Psi}M)$. So we have
shown that
\[
M\iso A\otimes_{A^\Psi}(A\square_{\Psi}M).
\qedhere
\]
\end{proof}

\begin{rem}\label{rem:Unicursal-modules}
The above algebra can be interpreted scheme-theoretically as
follows. Given a flat morphism of affine schemes $f\:X\lra Y$,
$X\times_Y X$ becomes a groupoid scheme with a unique morphism
$u\ra v$ whenever $f(u)=f(v)$. Comodules for the representing
Hopf algebroid are equivalent to $\mathcal{O}_X$-modules with
descent data, and the category of such comodules is equivalent
to that of $\mathcal{O}_Y$-modules. See~\cite[section~17.2]{Waterhouse}
for an algebraic version of this when $f$ is faithfully flat.
\end{rem}

\begin{examp}\label{examp:Unicursal-Galextn}
Let $R$ be a commutative ring and let $G$ be a finite group
which acts faithfully on $R$ by ring automorphisms so that
$R^G\lra R$ is a $G$-Galois extension in the sense of~\cite{CHR}.
Thus there is an isomorphism of rings
\begin{equation}\label{eqn:Unicursal-Galextn}
R\otimes_{R^G}R \iso R\otimes_{R^G}R^GG^*,
\end{equation}
where the dual group ring is $R^GG^*=\Map(G,R^G)$. The left
hand side is visibly a Hopf algebroid and as an $R^G$-module,
$R$ is finitely generated projective, so
Lemma~\ref{lem:Unicursal-modules} applies.

Following the outline in~\cite{AB-LFT}, we can identify
\[
R\otimes_{R^G}R^GG^* \iso RG^*
\]
with the dual of the twisted group ring $R\sharp G$ and
thus it also carries a natural Hopf algebroid structure.
It is easy to verify that this structure agrees with that
on $R\otimes_{R^G}R$ under~\eqref{eqn:Unicursal-Galextn}.

Interpreting a $R\otimes_{R^G}R$-comodule $M$ as equivalent
to a $R\sharp G$-module, we can use the Galois theoretic
isomorphism $R\sharp G\iso\End_{R^G}R$ to show that there
is an isomorphism of $R\sharp G$-modules
\[
M \iso R\otimes_{R^G}M^G,
\]
and since
\[
M^G\iso R\square_{R\otimes_{R^G}R}M,
\]
this is a module theoretic interpretation of the comodule
result of Lemma~\ref{lem:Unicursal-modules}.
\end{examp}

Now we recall some facts from \cite[lemma~A1.1.13]{DCR:Book2}
about the extension of Hopf algebroids
\[
(D,\Phi) \lra (A,\Gamma) \lra (A,\Gamma'),
\]
where $\Gamma'$ is the Hopf algebra associated to $\Gamma$ and
$\Phi$ is unicursal. We have the following identifications:
\[
\Gamma' = A\otimes_{\Phi} \Gamma,
\quad
   D    = A\square_\Gamma A,
\quad
   \Phi = A\otimes_D A.
\]
The map of Hopf algebroids $\Gamma\lra\Gamma'$ is normal and
\[
\Phi = \Gamma\square_{\Gamma'}A = A\square_{\Gamma'}\Gamma\subseteq\Gamma.
\]
Furthermore, for any left $\Gamma$-comodule $M$, $A\square_{\Gamma'}M$
is naturally a left $\Phi$-comodule and there is an isomorphism
of $A$-modules
\begin{equation}\label{eqn:DCRA1.1.13}
A\square_{\Gamma}M \iso A\square_{\Phi}(A\square_{\Gamma'}M).
\end{equation}

\begin{prop}\label{prop:Prim=/0}
Let $M$ be a $\Gamma$-comodule. If when viewed as a
$\Gamma'$-comodule, $M$ has non-trivial primitive
$\Gamma'$-subcomodule $A\square_{\Gamma'}M$, then the primitive
$\Gamma$-subcomodule $A\square_{\Gamma}M$ is non-trivial.
\end{prop}
\begin{proof}
Combine Lemma~\ref{lem:Unicursal-modules} and~\eqref{eqn:DCRA1.1.13}.
\end{proof}

Our next result is immediate.
\begin{thm}\label{thm:Unicursal-H->HA}
Let $(\boldsymbol{k},\Gamma)$ is a Hopf algebroid over a field.
If the associated Hopf algebra $(\boldsymbol{k},\Gamma')$ is
unipotent, then $(\boldsymbol{k},\Gamma)$ is unipotent.
\end{thm}

\section{Lubin-Tate spectra and their Hopf algebroids}
\label{sec:L-T}

In this section we will discuss the case of a Lubin-Tate spectrum~$E$
and its associated Hopf algebroid $(E_*,E^\vee_*E)$, where~$E$ denotes
any of the $2$-periodic spectra lying between $\hat{\mathcal{E}(n)}$
(by which we mean the $2$-periodic version of the completed
$2(p^n-1)$-periodic Johnson-Wilson spectrum $\hat{E(n)}$) and $\Enr_n$
discussed in~\cite{AB&BR}, see especially section~7. The most important
case is the `usual' Lubin-Tate spectrum $E_n$ for which
\[
\pi_*(E_n) = W\F_{p^n}[[u_1,\ldots,u_{n-1}]][u^{\pm 1}],
\]
but other examples are provided by the $K(n)$-local Galois subextension
of $\Enr_n$ over $\hat{\mathcal{E}(n)}$ in the sense of Rognes~\cite{JR:Galois}.
In all cases, $E_*=\pi_*(E)$ is a local ring with maximal ideal induced
from that of $\hat{E(n)}_*$, and we will write $\mathfrak{m}$ for this.
The residue field $E_*/\mathfrak{m}$ is always a graded subfield of the
algebraic closure $\bar{\F}_p[u,u^{-1}]$ of $\F_p[u,u^{-1}]$.

\begin{rem}\label{rem:Gradings}
Since all of the spectra considered here are $2$-periodic we will
sometimes treat their homotopy as $\Z/2$-graded and as it is usually
trivial in odd degrees, we will often focus on even degree terms.
However, when discussing reductions modulo a maximal ideal, it is
sometimes more useful to regard the natural periodicity as having
degree $2(p^n-1)$ with associated $\Z/2(p^n-1)$-grading; more
precisely, we will follow the ideas of~\cite{CGO} and take consider
gradings on $\Z/2(p^n-1)$ together with the non-trivial bilinear
pairing
\[
\nu\:\Z/2(p^n-1)\times\Z/2(p^n-1) \lra \{1,-1\};
\quad
\nu(\bar{i},\bar{j}) = (-1)^{ij},
\]
where $\bar{i}$ denotes the residue class $i\pmod{2(p^n-1)}$.
\end{rem}

We will denote by $K=E\wedge_{\hat{E(n)}}K(n)$ the version of
Morava $K$-theory associated to~$E$, it is known that $E$ is
$K$-local in the category of $E$-modules and we can consider
the localisation $L_K(E\wedge E)$ for which
\[
E^\vee_*E = \pi_*(L_K(E\wedge E)).
\]
By~\cite[proposition~2.2]{Ho:colim}, this localisation can be
taken either with respect to $K$ in the category of $S$-modules,
or with respect to $E\wedge K$ in the category of $E$-modules.
By~\cite[lemma~7.6]{AB&BR}, the homotopy $\pi_*(L_K M)$ viewed
as a module over the local ring $(E_*,\mathfrak{m})$ is
$L$-complete.

We will write $\Map(X,Y)$ for the set of all functions $X\lra Y$
and $\Mapc(X,Y)$ for the set of all continuous functions if $X,Y$
are topologised.

A detailed discussion of the relevant $K(n)$-local Galois theory
of Lubin-Tate spectra can be found in section~5.4 and chapter~8
of~\cite{JR:Galois}, and we adopt its viewpoint and notation. In
particular, $\Enr_n$ is a $K(n)$-local Galois extension of $L_{K(n)}$
with profinite Galois group
\[
\Gnr_n = \hZ \ltimes \mathbb{S}_n,
\]
where $\mathbb{S}_n$ is the usual Morava stabiliser group which
can be viewed as the full automorphism group of a height~$n$
Lubin-Tate formal group law $F_n$ defined over $\F_{p^n}\subseteq\Fc_p$,
and also as the group of units in the maximal order of a central
division algebra over $\Q_p$ of Hasse invariant $1/n$. The $p$-Sylow
subgroup $\mathbb{S}^0_n\lhd\mathbb{S}_n$ has index $(p^n-1)$ and
$\mathbb{S}_n$ is the semi-direct product
\[
\mathbb{S}_n = \F_{p^n}^\times\ltimes\mathbb{S}^0_n.
\]
The profinite group $\hZ$ acts as the Galois group
\[
\Gal(W\Fc_p/W\F_p)\iso\Gal(\Fc_p/\F_p)\iso\hZ.
\]
In particular, the closed subgroup $n\hZ\lhd\hZ$ is the stabiliser
of $\F_{p^n}$ and $E_n\simeq(\Enr_n)^{h(n\hZ)}$; similarly,
$\hat{\mathcal{E}(n)}\simeq(\Enr_n)^{h\hZ}$.

Our first result is a generalisation of a well
known result, see~\cite{AB-LFT} for example.
\begin{thm}\label{thm:Ev*E-contfns}
For $E$ as above, there are natural isomorphisms of $E_0$-algebras
\[
E^\vee_*\hat{\mathcal{E}(n)} \iso \Mapc(\mathbb{S}_n,E_*).
\]
Furthermore, $\Mapc(\mathbb{S}_n,E_*)$ is a pro-free $L$-complete
$E_*$-module.
\end{thm}
\begin{thm}\label{thm:Ev*E/m}
Let $E$ be a Lubin-Tate spectrum as above. \\
\emph{(a)} $(E_*,E^\vee_*E)$ is an $L$-complete Hopf algebroid. \\
\emph{(b)} the maximal ideal $\mathfrak{m}\lhd E_*$ is invariant. \\
\emph{(c)} $E^\vee_*E$ is a pro-free $E_*$-module. \\
\emph{(d)} There are isomorphisms of $K_*=E_*/\mathfrak{m}$-algebras
\[
K_*E \iso E^\vee_*E/E^\vee_*E\mathfrak{m} \iso
E_*/\mathfrak{m}[\theta_k:k\geq1]/
           (\theta_\ell^{p^n}-u^{p^\ell-1}\theta_\ell:\ell\geq1)
                       \otimes_{\F_p[u,u^{-1}]}E_*/\mathfrak{m}.
\]
\end{thm}

Now let us consider the reduction $K_*E$ in greater detail.
First note that the pair $(K_*,K_*E)$ is a $\Z$-graded Hopf
algebroid. Now
\[
K_* = \F[u,u^{-1}],
\]
where $\F\subseteq\Fc_p$ and $|u|=2$. Since $u^{p^n-1}=v_n$
under the map $BP\lra K$ classifying a complex orientation,
$u^{p^n-1}$ is invariant. This suggests that we might usefully
change to a $\Z/2(p^n-1)$-grading on $K_*$-modules by setting
$u^{p^n-1}=1$. To emphasise this regrading we write $(-)_\bullet$
rather than $(-)_*$. In particular, $K_\bullet = \F(u)$.

The right unit generates a second copy of $K_\bullet$
in $K_\bullet E$ and there is an element
\[
\theta_0 = \eta_L(u)^{-1}\eta_R(u)
\]
which satisfies the relation
\[
\theta_0^{p^n-1} = 1.
\]
The coproduct makes $\theta_0$ group-like,
\[
\psi(\theta_0) = \theta_0\otimes\theta_0.
\]
Now it is easy to see that $K_\bullet E$ contains
the unicursal Hopf algebroid
\begin{equation}\label{eqn:K_*E-unicursal}
K_\bullet\otimes_{\F_p}K_\bullet
=
\F\otimes_{\F_p}\F(u,\theta_0)
=
\F\otimes_{\F_p}\F[u,\theta_0]/(u^{p^n-1}-1,\theta_0^{p^n-1}-1),
\end{equation}
where $\theta_0,u$ have degrees $\bar{0},\bar{2}\in\Z/2(p^n-1)$
respectively.

Ignoring the generator $u$ and the grading, we also have
the ungraded Hopf algebroid
\[
(\F,\F\otimes_{\F_p}\F(\theta_0))
=
(\F,\F\otimes_{\F_p}\F[\theta_0]/(\theta_0^{p^n-1}-1))
\]
which is a subHopf algebroid of $(\F,K_{\bar{0}}E)$.

Since $\F$ is a Galois extension of $\F_p$ with Galois group
a quotient of $\hZ$, we obtain a ring isomorphism
\[
\F\otimes_{\F_p}\F \iso \F\Gal(\F/\F_p)^*.
\]
If $\Gal(\F/\F_p)$ is finite this has its usual meaning,
while if $\Gal(\F/\F_p)$ is infinite we have
\[
\F\Gal(\F/\F_p)^* = \Mapc(\Gal(\F/\F_p),\F).
\]
Of course, if $\Gal(\F/\F_p)$ is finite this interpretation
is still valid but then all maps $\Gal(\F/\F_p)\lra\F$ are
continuous. In each case, we obtain an isomorphism of Hopf
algebroids
\begin{equation}\label{eqn:K_*E-unicursalDualGroup}
\F\otimes_{\F_p}\F(\theta_0)
      \iso \Map(\Gal(\F/\F_p)\ltimes\F_{p^n}^\times,\F).
\end{equation}

Now we consider the associated Hopf algebra over the graded
field $K_\bullet$,
\[
K_\bullet\otimes_{\F\otimes_{\F_p}\F(u,\theta_0)}K_\bullet E
=
K_\bullet[\theta_k:k\geq1]/(\theta_k^{p^n}-\theta_k:k\geq1)
\]
whose zero degree part is
\begin{equation}\label{eqn:K0E}
\F[\theta_k:k\geq1]/(\theta_k^{p^n}-\theta_k:k\geq1)
                            \iso \Mapc(\mathbb{S}_n^0,\F).
\end{equation}
The right hand side fits into the framework of
Example~\ref{examp:dualpro-groupring}, so this Hopf algebra
over $\F$ is unipotent. Tensoring up with $K_\bullet$ we
have the following graded version.
\begin{thm}\label{thm:K*E-unipotent}
The Hopf algebra
$(K_\bullet,K_\bullet\otimes_{\F\otimes_{\F_p}\F(u,\theta_0)}K_\bullet E)$
is unipotent.
\end{thm}

\begin{rem}\label{rem:K*E-Maps}
The identification of~\eqref{eqn:K0E} can be extended to all
degrees of
$K_\bullet\otimes_{\F\otimes_{\F_p}\F(u,\theta_0)}K_\bullet E$.
To make this explicit, we consider $\Mapc(\mathbb{S}_n,\F u^r)$
with the action of $\Gal(\F/\F_p)\ltimes\F_{p^n}^\times$ induced
from the action on $\mathbb{S}_n$ used in defining $\Gnr_n$ and
the $\F$-semilinear action of $\Gal(\F/\F_p)\ltimes\F_{p^n}^\times$
on $\F u^r$ obtained by inducing up the $r$-th power of the natural
$1$-dimensional representation of $\F_{p^n}^\times$. Then
\[
[K_\bullet\otimes_{\F\otimes_{\F_p}\F(u,\theta_0)}K_\bullet E]_{\widebar{2r}}
\iso
\Mapc(\mathbb{S}_n,\F u^r)^{\Gal(\F/\F_p)\ltimes\F_{p^n}^\times},
\]
where the right hand side is the set of continuous
$\Gal(\F/\F_p)\ltimes\F_{p^n}^\times$-equivariant maps. This
is essentially a standard identification appearing in work
of Morava and others in the 1970's.
\end{rem}

Combining Theorems~\ref{thm:K*E-unipotent} and~\ref{thm:Main0}
we obtain our final result in which revert to $\Z$-gradings.

\begin{thm}\label{thm:K*E-HA-unipotent}
The Hopf algebroid $(K_*,K_*E)$ is unipotent, hence every
finitely generated comodule for the $L$-complete Hopf
algebroid $(E_*,E^\vee_*E)$ has a Landweber filtration.
\end{thm}

Here it is crucial that we take proper account of the grading
since the ungraded Hopf algebra $(K_0,K_0E)$ is not unipotent:
this can be seen by considering the comodule $K_0S^2$ which
is not isomorphic to $K_0S^0$.

\appendix
\section{Representations of Galois Hopf algebroids}
\label{sec:RepGalHA}

Twisted (or skew) group rings are standard algebraic objects.
They were discussed for an audience of topologists in~\cite{AB-LFT},
and their duals as Hopf algebroids were discussed. For a
recent reference on their modules see~\cite{Kuenzer}. Here we
focus on the special case of Galois extensions of fields, which
is closely related to the unicursal Hopf algebroids. In particular,
the unicursal Hopf algebroids associated with $K_\bullet$ in
Section~\ref{sec:L-T} contain degree zero parts of this form.

Let $\boldsymbol{k}$ be a field of positive characteristic
$\Char\boldsymbol{k}=p$ and let $A$ be a (finite dimensional)
commutative $\boldsymbol{k}$-algebra which is a $G$-Galois
extension of $\boldsymbol{k}$ for some finite group $G$, where
the action of $\gamma\in G$ on $x\in A$ is indicated by writing
${}^\gamma x$. This means that
\begin{itemize}
\item
$A^G = \boldsymbol{k}$,
\item
the $A$-algebra homomorphism
\[
A\otimes_{\boldsymbol{k}}A \lra \prod_{\gamma\in G}A;
\quad
x\otimes y \mapsto (x{}^\gamma y)_{\gamma\in G}
\]
is an isomorphism, where the $A$-algebra comes from the left
hand factor of $A$.
\end{itemize}
The second condition is equivalent to the assertion that there
is an isomorphism of $\boldsymbol{k}$-algebras
\begin{equation}\label{eqn:GalIso-Dual}
A\otimes_{\boldsymbol{k}}A
     \iso A\otimes_{\boldsymbol{k}}\boldsymbol{k}G^*,
\end{equation}
where
\[
\boldsymbol{k}G^*
      =\Hom_{\boldsymbol{k}}(\boldsymbol{k}G,\boldsymbol{k})
\]
is the dual group algebra.

The \emph{twisted group ring} $A\sharp G$ is the usual group
ring $AG$ as a left $A$-module, but with multiplication
defined by
\[
(a_1\gamma_1)(a_2\gamma_2) = a_1{}^{\gamma_1}a_2\gamma_1\gamma_2.
\]
There is a natural $\boldsymbol{k}$-linear map
\[
A\sharp G \lra \End_{\boldsymbol{k}}A
\]
under which $a\gamma\in A\sharp G$ is sent to the
$\boldsymbol{k}$-linear endomorphism $x \mapsto a{}^\gamma x$.
Another consequence of the above assumptions is that this is
an $\boldsymbol{k}$-algebra isomorphism, see~\cite{CHR}.

If $A=\boldsymbol{\ell}$ is a field, then using the isomorphism
of~\eqref{eqn:GalIso-Dual} we see that
$\boldsymbol{\ell}\otimes_{\boldsymbol{k}}\boldsymbol{\ell}$
is isomorphic to $\boldsymbol{\ell}G^*$ as an
$\boldsymbol{\ell}$-algebra. There is an associated `right'
action of $\boldsymbol{\ell}$ on $\boldsymbol{\ell}G^*$ given
by
\[
(f\.x)(\gamma) = {}^\gamma xf(\gamma)
\]
for $f\in\boldsymbol{\ell}G^*$, $x\in\boldsymbol{\ell}$
and $\gamma\in G$.

A proof of the next result is sketched in~\cite{AB-LFT}.
\begin{prop}\label{prop:TwistGpRng-HA}
The pair $(\boldsymbol{\ell},\boldsymbol{\ell}G^*)$ is
a Hopf algebroid.
\end{prop}
\begin{lem}\label{lem:TwGpRg-ss}
The twisted group ring $\boldsymbol{\ell}\sharp G$ is
a simple $\boldsymbol{k}$-algebra and every finite
dimensional $\boldsymbol{\ell}\sharp G$-module $V$ is
completely reducible. In particular, if $V\neq 0$ then
$V^G\neq0$ and there is an $\boldsymbol{\ell}$ linear
isomorphism
\[
\boldsymbol{\ell}\otimes_{\boldsymbol{k}}V^G \lra V;
\quad
x\otimes v \mapsto xv.
\]
\end{lem}
\begin{proof}
Since $\End_{\boldsymbol{k}}A$ is an irreducible
$\boldsymbol{k}$-algebra, it has a unique simple module
which agrees with $A$ as a $\boldsymbol{k}$-module. Hence
every finite dimensional module is isomorphic to a direct
sum of copies of $A$. Since $A^G=\boldsymbol{k}$, we see
that $V^G\neq0$. Verifying the bijectivity of the linear
map is straightforward.
\end{proof}

We can generalise this situation and still get similar
results. For example, if $\tilde G$ is a finite group
with a given epimorphism $\pi\:\tilde G\lra G$, then
$\boldsymbol{\ell}\sharp\tilde G$ is semi-simple provided
that $p\nmid |\ker\pi|$, see~\cite{Kuenzer}. In fact the
unicursal Hopf algebroid $\F\otimes_{\F_p}\F(\theta_0)$
of~\eqref{eqn:K_*E-unicursalDualGroup} is dual to
\[
\F\sharp(\Gal(\F/\F_p)\ltimes\F_{p^n}^\times),
\]
where the action of the group on $\F$ is through the
projection onto $\Gal(\F/\F_p)$.


\section{Non-exactness of tensoring with a pro-free module}\label{sec:App2}

In \cite[section~1]{Ho:colim}, it was shown that in $\hat{\mathscr{M}}$,
coproducts need not preserve left exactness. At the suggestion of the
referee, we include a more precise example showing that tensoring with
a pro-free module in $\hat{\mathscr{M}}$ need not be a left exact functor.
This material is due to the referee to whom we are grateful for the
opportunity to include it. The main result is Theorem~\ref{thm:App}, but
we need several preparatory technical results.
\begin{lem}\label{lem:App-1}
Let
\[
M_0\xleftarrow{\;f\;}M_1\xleftarrow{\;f\;}M_2\xleftarrow{\;f\;}\cdots
\]
be a inverse system of abelian groups for which
$\ds\lim_n M_n = 0 = {\lim_n}^1 M_n$ and where each homomorphism
$M_n\lra M_0$ is non-zero. Then the induced inverse system
\[
\bigoplus_{k\in\N} M_0\xleftarrow{\;\tilde f\;}\bigoplus_{k\in\N} M_1
\xleftarrow{\;\tilde f\;}\bigoplus_{k\in\N} M_2\xleftarrow{\;\tilde f\;}\cdots
\]
satisfies $\ds{\lim_n}^1\bigoplus_{k\in\N} M_n\neq0$.
\end{lem}
\begin{proof}
For ease of notation we write $\bigoplus_{k}$ for $\bigoplus_{k\in\N}$
and $\prod_{n}$ for $\prod_{n\in\N_0}$, where $\N=\{0\}\cup\N$.

Consider the commutative square
\[
\xymatrix{
\prod_n\bigoplus_{k} M_n\ar[r]^d\ar@{ >->}[d]
                & \prod_n\bigoplus_{k} M_n\ar@{ >->}[d] \\
\prod_n\prod_{k} M_n\ar[r]^{d'}_\iso & \prod_n\prod_{k} M_n
}
\]
in which $d$ is the shift map with
\[
\ker d = {\lim_n}^1(\bigoplus_{k} M_n),
\quad
\coker d = \lim_n(\bigoplus_{k} M_n),
\]
and similarly for $d'$. The vanishing of $\ds{\lim_n}^s M_n$
for $s=0,1$ implies that $d'$ is an isomorphism.

Now choose a sequence of elements $a_n\in M_n$ with non-zero
images in $M_0$. Define
\[
b_{nk}=
\begin{cases}
a_n & \text{if $k=n$}, \\
\;0   & \text{if $k\neq n$},
\end{cases}
\]
and let $b=(b_{n,k})$ be the resulting element of $\prod_n\prod_kM_n$.
Defining $c=(d')^{-1}(b)$, we see that
\begin{equation}\label{eqn:b-nk}
b_{nk} = c_{nk} - f(c_{n+1\,k})
\end{equation}
for all $n,k$. Now fix $k$ and consider the $c_{nk}$ for
$n>k$; these satisfy
\[
c_{nk} = f(c_{n+1\,k}),
\]
hence they yield an element of inverse limit of the inverse
system
\[
M_k\xleftarrow{\;f\;}M_{k+1}\xleftarrow{\;f\;}M_{k+2}\xleftarrow{\;f\;}\cdots,
\]
but
\[
\ds\lim_{n\geq k}M_n = \lim_{n}M_n = 0.
\]
Therefore $c_{nk}=0$ for $n>k$. Using~\eqref{eqn:b-nk},
we see that for all $k\geq n$,
\[
c_{nk} = f^{k-n}(a_k)
\]
and in particular, in $M_0$,
\[
c_{0k} = f^k(a_k) \neq 0.
\]
This shows that $c\notin\prod_n\bigoplus_{k} M_n$ even
though $d'(c)=b\in\prod_n\bigoplus_{k} M_n$. The result
follows by inspection of the diagram.
\end{proof}

Now let $R=\Z_p[[u]]$ with $\mathfrak{m}=(p,u)$ its maximal
ideal containing $p$ and $u$. Let $M$ be the
$\mathfrak{m}$-adic completion of $\bigoplus_nR$ which
can be identified with the set of sequences $x=(x_n)\in\prod_nR$
for which $x_n\ra 0$ in the $\mathfrak{m}$-adic topology.
The group $N=\prod_n(p^n,u^n)$ is a subgroup of $M$.

The category of $L$-complete modules is closed under products
and contains all finitely generated $R$-modules, therefore~$N$
and $M/N$ are $L$-complete.
\begin{lem}\label{lem:App-2}
For each $n\geq1$, the natural map
\[
\Tor^R_2(R/\mathfrak{m}^n,M/N)\lra\Tor^R_2(R/\mathfrak{m},M/N)
\]
is non-zero.
\end{lem}
\begin{proof}
Since $R/(p^n,u^n)$ is a retract of $N$, it suffices to prove
the result for $R/(p^n,u^n)$ in place of $N$. The sequence
$p^n,u^n$ is regular in $R$, so for any $R$-module~$K$ we can
compute $\Tor^R_*(R/(p^n,u^n),K)$ using a Koszul resolution.
In particular, if $p^nK=0=u^nK$ then we have
\[
\Tor^R_*(R/(p^n,u^n),K) = K
\]
and the reduction map is the obvious epimorphism
$R/(p^n,u^n)\lra R/(p,u)$. The result follows easily from this.
\end{proof}

\begin{cor}\label{cor:App-2}
The module $L_1(\bigoplus_k M/N)$ is non-zero.
\end{cor}
\begin{proof}
For each $s\geq1$, the natural short exact sequence
of~\eqref{eqn:Ls-exactseq} and the fact that $M/N$ is
$L$-complete and so $L_sM/N=0$, together yield
\[
\lim_n\Tor^R_2(R/\mathfrak{m}^n,K) =0=
                {\lim_n}^1\Tor^R_2(R/\mathfrak{m}^n,K).
\]
This gives one of the hypotheses of Lemma~\ref{lem:App-1},
and Lemma~\ref{lem:App-2} gives the other. Therefore
\[
{\lim_n}^1\Tor^R_2(R/\mathfrak{m}^n,M/N) \neq 0.
\]
Now applying~\eqref{eqn:Ls-exactseq} to $M/N$ gives
$L_1(\bigoplus_kM/N)\neq0$.
\end{proof}

\begin{lem}\label{lem:App-3}
If the sequence $p,u$ acts regularly on the $R$-module $K$,
then $L_sK=0$ for $s>0$.
\end{lem}
\begin{proof}
Using the exact sequence~\eqref{eqn:Ls-exactseq}, it suffices
to show that for all $s>0$, $\Tor^R_s(R/\mathfrak{m}^n,K)=0$.
This can be deduced from the case $n=1$ since $R/\mathfrak{m}^n$
has a composition series with simple quotient terms isomorphic
$R/\mathfrak{m}$. This case $n=1$ can be directly verified using
the Koszul resolution.
\end{proof}

Here is the main result of this Appendix which complements
an example of~\cite{Ho:colim}.
\begin{thm}\label{thm:App}
The natural map $L_0(\bigoplus_kN)\lra L_0(\bigoplus_kM)$
is not injective.
\end{thm}
\begin{proof}
The short exact sequence
\[
0\ra \bigoplus_kN\lra \bigoplus_kM\lra \bigoplus_kM/N \ra0
\]
induces an exact sequence
\[
L_1(\bigoplus_kM)\lra L_1(\bigoplus_kM/N)
   \lra L_0(\bigoplus_kN)\lra L_0(\bigoplus_kM)\ra0.
\]
The sequence $p,u$ acts regularly on $\bigoplus_kM$, so
Lemma~\ref{lem:App-3} shows that $L_1(\bigoplus_kM)=0$,
while Corollary~\ref{cor:App-2} shows that
$L_1(\bigoplus_kM/N)\neq0$.
\end{proof}

\end{document}